\documentclass{amsart}

\usepackage{amssymb,amsmath,latexsym,amsthm}
\usepackage{MnSymbol}
\usepackage{graphicx}
\usepackage{fontenc}%
\usepackage{mathrsfs}
\usepackage{color}

\usepackage[english]{babel} 
\usepackage{pstricks} 
\usepackage{pstricks-add} 

\newtheorem{theorem}{Theorem}[section]
\newtheorem{definition}[theorem]{Definition}
\newtheorem{lemma}[theorem]{Lemma}

\newtheorem{example}[theorem]{Example}

\newtheorem{remark}[theorem]{Remark}

\def\defineTColor#1#2{%
 \newpsstyle{#1}{%
  fillstyle=vlines,hatchcolor=#2,
  hatchwidth=0.1\pslinewidth,
  hatchsep=1\pslinewidth}%
  }
\defineTColor{Tgray}{gray}
\defineTColor{Tgreen}{green}
\defineTColor{Tyellow}{yellow}
\defineTColor{Tred}{red} 
\defineTColor{Tblue}{blue} 
\defineTColor{Tmagenta}{magenta}
\defineTColor{Torange}{orange}
\defineTColor{Twhite}{white}
\defineTColor{Tgray}{lgrey}
\defineTColor{Tbgreen}{brightgreen}

\newcommand{\norm}[1]{\left\|#1\right\|}  
\newcommand{\sn}{\mathop{\delta}\limits^{\doublewedge}}
\newcommand{\notsn}{\mathop{\not\delta}\limits^{\doublewedge}}
\newcommand{\snd}{\mathop{\delta_{_{\Phi}}}\limits^{\doublewedge}} 
\newcommand{\notsnd}{\mathop{\not\delta_{_{\Phi}}}\limits^{\doublewedge}} 
\newcommand{\sndatlases}{\mathop{\delta} \limits^{\doublewedge}_{_{\hat{\mathcal{A}},\Phi}}} 
\newcommand{\CL}{\mbox{CL}}
\newcommand{\cl}{\mbox{cl}}
\newcommand{\Int}{\mbox{int}}

\begin{document}

\title[Strong Proximities on Smooth Manifolds and Vorono\" i Diagrams]{Strong Proximities on Smooth Manifolds and Vorono\" i Diagrams}

\author[J.F. Peters]{J.F. Peters$^{\alpha}$}
\email{James.Peters3@umanitoba.ca, cguadagni@unisa.it}
\address{\llap{$^{\alpha}$\,}Computational Intelligence Laboratory,
University of Manitoba, WPG, MB, R3T 5V6, Canada and
Department of Mathematics, Faculty of Arts and Sciences, Ad\i yaman University, 02040 Ad\i yaman, Turkey}
\author[C. Guadagni]{C. Guadagni$^{\beta}$}
\address{\llap{$^{\beta}$\,}Computational Intelligence Laboratory,
University of Manitoba, WPG, MB, R3T 5V6, Canada and
Department of Mathematics, University of Salerno, via Giovanni Paolo II 132, 84084 Fisciano, Salerno , Italy}
\thanks{The research has been supported by the Natural Sciences \&
Engineering Research Council of Canada (NSERC) discovery grant 185986.}


\date{}

\dedicatory{Dedicated to the Memory of Som Naimpally}

\begin{abstract}
This article introduces strongly near smooth manifolds.  The main results are  (i)  second countability of the strongly hit and far-miss topology on a family $\mathcal{B}$ of subsets  on the Lodato proximity space of regular open sets to which singletons are added, (ii) manifold strong proximity, (iii) strong proximity of charts in manifold atlases implies that the charts have nonempty intersection.  The application of these results is given in terms of the nearness of atlases and charts of proximal manifolds and what are known as Vorono\" i manifolds.
\end{abstract}


\maketitle

\section{Introduction}
This article carries forward recent work on strong proximities~\cite{Peters2015VoronoiAMSJ,Peters2015visibility,PetersGuadagni2015stronglyNear,PetersGuadagni2015connectedness} and their applications~\cite{Hettiarachchi2015PatRec,Peters2015KBS}, which is a direct result of work on proximity~\cite{Beer2015Som,Beer1993hit,DiConcilio2009,DiConcilio2000SetOpen,DiConcilio2000PartialMaps,DiMaio1995hypertop,Lodato1962,Naimpally1970,Naimpally2002,Naimpally2009,Naimpally2013,Peters2016CP}.  Applications of the results in this paper are given in terms of the atlases and charts of proximal manifolds and what are known as Vorono\" i manifolds, which reflect recent work on manifolds~\cite{Hettiarachchi2015PatRec,Peters2014manifolds}.

\section{Preliminaries}
The concept of \emph{strong proximity} is characterized by a relation giving information about pairs of sets that share points.  Such proximities are not the usual proximities. In fact, in the traditional sense, proximal sets do not always have points in common. Actually, the name \emph{strong proximity} signals a strong kind of nearness between sets with points in common.  

\begin{definition}
Let $X$ be a topological space, $A, B, C \subset X$ and $x \in X$.  The relation $\sn$ on $\mathscr{P}(X)$ is a \emph{strong proximity}, provided it satisfies the following axioms.
\begin{enumerate}
\item[(N0)] $\emptyset \not\sn A, \forall A \subset X $, and \ $X \sn A, \forall A \subset X$
\item[(N1)] $A \sn B \Leftrightarrow B \sn A$
\item[(N2)] $A \sn B \Rightarrow A \cap B \neq \emptyset$
\item[(N3)] If $\{B_i\}_{i \in I}$ is an arbitrary family of subsets of $X$ and  $A \sn B_{i^*}$ for some $i^* \in I \ $ such that $\Int(B_{i^*})\neq \emptyset$, then $  \ A \sn (\bigcup_{i \in I} B_i)$
\item[(N4)] $\mbox{int}A \cap \mbox{int} B \neq \emptyset \Rightarrow A \sn B$ \qquad \textcolor{blue}{$\blacksquare$}
\end{enumerate}
\end{definition}

\noindent When we write $A \sn B$, we read $A$ is \emph{strongly near} $B$.  The notation $A\ \notsn\ B$ reads $A$ is not strongly near $B$.
For each \emph{strong proximity}, we assume the following relations:
\begin{enumerate}
\item[(N5)] $x \in \Int (A) \Rightarrow x \sn A$ 
\item[(N6)] $\{x\} \sn \{y\} \Leftrightarrow x=y$  \qquad \textcolor{blue}{$\blacksquare$} 
\end{enumerate}
So, for example, if we take the strong proximity related to non-empty intersection of interiors, we have that $A \sn B \Leftrightarrow \Int A \cap \Int B \neq \emptyset$ or either $A$ or $B$ is equal to $X$, provided $A$ and $B$ are not singletons; if $A = \{x\}$, then $x \in \Int(B)$, and if $B$ too is a singleton, then $x=y$. It turns out that if $A \subset X$ is an open set, then each point that belongs to $A$ is strongly near $A$.

Related to this new kind of nearness introduced in~\cite{PetersGuadagni2015stronglyNear} which extends traditional proximity (see, {\em e.g.},~\cite{Naimpally1970,Lodato1962,Lodato1964,Lodato1966,Naimpally2013,Peters2012notices}), we defined a new kind of \emph{hit-and-miss hypertopology}, \cite{PetersGuadagni2015stronglyNear, PetersGuadagni2015hypertopologies}, which extends recent work on hypertopologies (see, {\em e.g.},~\cite{Beer1993hit,DiConcilio1989,DiConcilio2013action,DiMaio2008hypertop,DiMaio1995hypertop,DiMaio1992hypertop,Som2006hypertopology,Guadagni2015,Naimpally2002}).  The important thing to notice is that this work has its foundation in geometry~\cite{Guadagni2015,Peters2015VoronoiAMSJ,Peters2015visibility}.

The \emph{strongly hit and far-miss} topology $\tau^\doublewedge_\mathscr{B}$ associated to $\mathscr{B}$ has as subbase the sets of the form:
\begin{enumerate}
\item $V^{\doublewedge} = \{E \in \CL(X): E \sn V \}$, where $V$ is an open subset of $X$,
\item $A^{++} =  \{ \ E \in \CL(X) : E \not\delta X\setminus  A  \ \}$, where $A$ is an open subset of $X$ and $X \setminus A \in \mathscr{B}$.
\end{enumerate} 

In the definition of $A^{++}$, $\delta$ represents a Lodato proximity.
\begin{definition} 
Let $X$ be a nonempty set. A \textit{Lodato proximity $\delta$} is a relation on $\mathscr{P}(X)$, which satisfies the following properties for all subsets $A, B, C $ of $X$:
\begin{enumerate}
\item[(P0)] $A\ \delta\ B \Rightarrow B\ \delta\ A$
\item[(P1)] $A\ \delta\ B \Rightarrow A \neq \emptyset $ and $B \neq \emptyset $
\item[(P2)] $A \cap B \neq \emptyset \Rightarrow  A\ \delta\ B$
\item[(P3)] $A\ \delta\ (B \cup C) \Leftrightarrow A\ \delta\ B $ or $A\ \delta\ C$
\item[(P4)] $A\ \delta\ B$ and $\{b\}\ \delta\ C$ for each $b \in B \ \Rightarrow A\ \delta\ C$
\end{enumerate}
Further $\delta$ is \textit{separated }, if 
\begin{enumerate}
\item[(P5)] $\{x\}\ \delta\ \{y\} \Rightarrow x = y$. \qquad \textcolor{blue}{$\blacksquare$}
\end{enumerate}
\end{definition}

\noindent $A\ \delta\ B$ reads "$A$ is near to $B$" and $A \not \delta B$ reads "$A$ is far from $B$".
\textit{Lodato proximity} or \textit{LO-proximity} is one of the simplest proximities. We can associate a topology with the space $(X, \delta)$ by considering as closed sets those sets that coincide with their own closure where. For a subset $A$, we have
\[
\mbox{cl} A = \{ x \in X: x\ \delta\ A\}.
\]
Any proximity $\delta$ on $X$ induces a binary relation over the powerset exp $X,$ usually denoted as $\ll_\delta $ and  named  the  {\it   natural strong inclusion associated with} $\delta$, by declaring that $ A$ is {\it strongly included} in $B, \ A \ll_{\delta} B$, when $A$ is far from the complement of $B,  \ A \not\delta X \setminus B .$

In a recent paper (\cite{PetersGuadagni2015hypertopologies}), we looked at the Hausdorffness of the hypertopology $\tau^\doublewedge_\mathscr{B}$.  Here, the focus is on second countability.

Moreover, we want to point out the real possibility to use this concepts in applications. For this reason we look at some kinds of descriptive strong proximities and strongly proximal Vorono\" i regions. 

\section{Second Countability of Strong Proximity Topology}

As for the Hausdorff property of $\tau^\doublewedge$, we concentrate our attention on the class of regular closed sets, $\mbox{RCL}(X)$.
Recall that a set $F$ is \emph{regular closed} if $F = \cl{(\Int F)}$, that is $F$ coincides with the closure of its interior.  A well-known fact is that regular closed sets form a complete Boolean lattice \cite{Ronse1990}. Moreover there is a one-to-one correspondence between regular open ($\mbox{RO}(X)$) and regular closed sets. We have a \emph{regular open} set $A$ when $A = \Int{(\cl A)}$, that is $A$ is the interior of its closure. The correspondence between the two mentioned classes is given by $c: \mbox{RO}(X) \rightarrow \mbox{RCL}(X)$, where $c(A)= \cl(A)$, and $o: \mbox{RCL}(X) \rightarrow \mbox{RO}(X)$, where $o(F) = \Int (F)$. By this correspondence it is possible to prove that also the family of regular open sets is a complete Boolean lattice. Furthermore it is shown that every complete Boolean lattice is isomorphic to the complete lattice of regular open sets in a suitable topology.\\
The importance of these families is also due to the possibility of using them for digital images processing, because they allow to satisfy certain common-sense physical requirements.\\

Consider now $\tau^\doublewedge_\mathscr{B}$, the \emph{strongly hit and far-miss} topology associated to a family  $\mathscr{B}$ of subsets of $X$, on the space of regular closed sets to which singletons are added, $\mbox{RCL}^*(X) = \mbox{RCL}(X) \cup \{ \{x\}:  x \in X \}$:

\begin{itemize}
\item $V^{\doublewedge} = \{E \in \mbox{RCL}^*(X): E \sn V \}$, where $V$ is a regular open subset of $X$,
\item $A^{++} =  \{ \ E \in \mbox{RCL}^*(X) : E \not\delta X\setminus  A  \ \}$, where $A$ is a regular open subset of $X$ and $X \setminus A \in \mathscr{B}$.
\end{itemize} 

The following theorem is a generalization of classical results holding for hit and miss hypertopologies, \cite{DiMaio1995hypertop, Zsilinszki}. In \cite{Zsilinszki}, L. Zsilinszki considers spaces that are \emph{weakly $R_0$}, {\em i.e.}, every nonempty difference of open sets contains a non-empty closed subset of $X$. We will use an analogous property that holds for regular open and regular closed sets.

\begin{definition}
We say that a topological space $X$ endowed with a compatible Lodato proximity $\delta$ is regularly weakly $R_0$, if and only if every nonempty difference of regular open sets  proximally contains a nonempty regular closed subset of $X$, that is 
\[ \forall A, B \in \mbox{RO}(X), \exists C \in \mbox{RCL}(X): C \ll_\delta (A \setminus B) . \]
\end{definition}

By $\Sigma(\mathscr{B})$ we indicate the set of all finite unions of members of $\mathscr{B}$.

\begin{theorem}\label{2count}
Let $X$ be a $T_1$, regularly weakly $R_0$ topological space, $\delta$ a compatible Lodato proximity on $X$, and $\sn$ a strong proximity on $X$. Then the following are equivalent:
\begin{itemize}
\item[i)] $( \mbox{RCL}^*(X), \tau^\doublewedge_\mathscr{B} )$ is second countable;
\item[ii)] $X$ is second countable and there exists a countable subfamily $\mathscr{B}' \subset \mathscr{B}$ such that for each $B \in \mathscr{B}$ and $A, B \in \mbox{RCL}^*(X)$ with $A \not\delta B$, then   $B \subset D \ll_\delta X \setminus A$ for some $D \in \Sigma(\mathscr{B'})$.
\end{itemize}
\end{theorem}

To prove Theorem~\ref{2count}, we need the following lemma.

\begin{lemma}\label{lemma1}
Let $X$ be a $T_1$ regularly weakly $R_0$ space, $U_1,...,U_n$, $V_1,..., V_m$ ($n,m \in \mathbb{N}$) regular open subsets of $X$, $B$ and $D$ regular closed sets belonging to $\Sigma( \mathscr{B})$. Then the following are equivalent:
\begin{itemize}
\item[a)] $\mathscr{U}= (\bigcap_{i=1}^n {U_i}^{\doublewedge}) \cap (X \setminus B)^{++} \subset \mathscr{V}= (\bigcap_{j=1}^m {V_i}^{\doublewedge}) \cap (X \setminus D)^{++} $
\item[b)] $X \setminus B \subset X \setminus D$ and for each $j \in \{1,..,m \}$ there exists $i \in \{1,..,n \}$ such that $U_i \cap (X \setminus B) \subset V_j \cap (X \setminus D)$
\end{itemize}
\end{lemma}

\begin{proof}
$(a) \Rightarrow (b)$ . Suppose $A \in \mathscr{U}$ and $(X \setminus B) \setminus (X \setminus D) \neq \emptyset $. Being $X$ regularly weakly $R_0$, there exists a regular closed set $C$ strongly included in $(X \setminus B) \setminus (X \setminus D)$. We want to prove that $A \cup C \in \mathscr{U} \setminus \mathscr{V}$. $A \cup C $ belongs to $\bigcap_{i=1}^n {U_i}^{\doublewedge}$ by property $(N3)$ of strong proximities; furthermore $A$ and $C$ are far from $B$ because $A \in (X \setminus B)^{++}$ and $C \ll_\delta X \setminus B$. Moreover $A \cup C \not\in \mathscr{V}$ because $C \subset (X \setminus B) \setminus (X \setminus D)$ means that $C \cap D \neq \emptyset$.

Now we want to prove the second part of $(b)$. Suppose, by contradiction, that there exists $j^* \in \{1,...,m \}$ such that for all $i \in \{1,..,n \} \ (U_i \cap (X \setminus B)) \setminus (V_{j^*} \cap (X \setminus D)) \neq \emptyset$. We use again the property of being regularly weakly $R_0$ for $X$ and we have that there exist regular closed subsets $A_i \ll_\delta (U_i \cap (X \setminus B)) \setminus (V_{j^*} \cap (X \setminus D))$. We claim that $\bigcup_{i=1}^n A_i \in \mathscr{U} \setminus \mathscr{V}$. Observe that $\bigcup_{i=1}^n A_i \in \mathscr{U}$  because of property $(N3)$ for strong proximities and property $(P3)$ for Lodato proximities. Instead, $\bigcup_{i=1}^n A_i \not\in \mathscr{V}$ because $A_i \ll_\delta (U_i \cap (X \setminus B)) \setminus (V_{j^*} \cap (X \setminus D))$ implies that $A_i \cap (V_{j^*} \cap (X \setminus D)) = \emptyset$ and, being $A_i \subset X \setminus B \subset X \setminus D$, we have that $A_i \cap V_{j^*}= \emptyset$ for all $i$. So by property $(N2)$ we have that $A_i \not\sn V_{j^*}$ and by property $(N3)$ we obtain $\bigcup_{i=1}^n A_i \not\sn V_{j^*}$.

$(b)\Rightarrow (a)$. Suppose that $A \in \mathscr{U}$ and $A \in \mbox{RCL}(X)$. We want to prove that $A$ belongs to $\mathscr{V}$ as well. Being $A \ll_\delta X \setminus B \subset X \setminus D$, we have that $A \not\delta D$. Moreover we have to prove that $A \sn V_j $ for each $j$. By the hypothesis we know that there exists $i$ such that $U_i \cap (X \setminus B) \subset V_j \cap (X \setminus D)$. So $A \cap U_i \subset A \cap V_j$. But, if $A \sn U_i$, then $A \cap U_i \neq \emptyset$, and by the regularity of $A$ we have that $\Int(A) \cap U_i \neq \emptyset$. Hence $\Int(A) \cap V_j \neq \emptyset$ and by property $(N4)$ we have $A \sn V_j$ for all $j$. If $A \in \mathscr{U}$ and $A$ is a point of $X$, the implication is easy. 
\end{proof}

Now we can prove Theorem~\ref{2count}.

\begin{proof}
\textit{(of thm. \ref{2count})}. $i) \Rightarrow ii)$. First of all we want to prove that $X$ is second countable. By $i)$ we know that there exist countable subfamilies $\mathscr{O} \subset \mbox{RO}(X)$ and $\mathscr{B}' \subset \mathscr{B}$ such that $\{ (\bigcap_{i=1}^n A_i^\doublewedge) \cap (X \setminus B) ^{++} : A_i \in \mathscr{O}, \ i\in \{1,..,n\}, \  n \in \mathbb{N}, \  X \setminus B \in \mathscr{O}, \ B \in \Sigma(\mathscr{B}') \}$ is a countable base for $ \tau^\doublewedge_\mathscr{B} $. We claim that $\{W \cap (X \setminus D): W, X \setminus D \in \mathscr{O}, D \in \Sigma(\mathscr{B}')\}$ is a countable base for the topology on $X$. Take any open set $V$ in $X$ and suppose $x \in V$. Choose $D \in \mathscr{B}$ such that $x \not\in D$. Then $x \in V^\doublewedge \cap (X \setminus D)^{++}$. So there exists an element of the countable base for $\tau^\doublewedge_\mathscr{B} $,  $(\bigcap_{i=1}^n A_i^\doublewedge) \cap (X \setminus B) ^{++} $, that contains $x$ and is contained in $V^\doublewedge \cap (X \setminus D)^{++}$. By lemma \ref{lemma1} we have that there exists $i^* \in \{1,..,n\}$ such that $A_{i^*} \cap (X \setminus B) \subset V \cap (X \setminus D)$. Hence $x \in A_{i^*} \cap (X \setminus B) \subset V \cap (X \setminus D)$ and the second countability is achieved.

Consider now $B \in \mathscr{B}$ and $B, K \in \mbox{RCL}^*(X)$ such that $B \not\delta K.$ So $K \in (X \setminus B)^{++}$ and, by $(i)$, there exists an element of the countable base for $ \tau^\doublewedge_\mathscr{B} $, $(\bigcap_{i=1}^n A_i^\doublewedge) \cap (X \setminus H) ^{++} $, that contains $K$ and is contained in $(X \setminus B)^{++}$. Hence, by lemma \ref{lemma1}, we have that $B \subset H$, where $H \in \Sigma(\mathscr{B}')$. Finally we have $B \subset H \ll_\delta X \setminus K$, being $K \in (X \setminus H) ^{++} $.

$ii) \Rightarrow i)$. Let $\mathfrak{T}$ be a countable base for $X$. Take any open set in $ \tau^\doublewedge_\mathscr{B} $, \ $\mathscr{U}=(\bigcap_{i \in I} V_i^\doublewedge) \cap (X \setminus C) ^{++} $, where $C \in \Sigma(\mathscr{B})$. Suppose $A \in \mathscr{U}$, with $A \in \mbox{RCL}(X)$. Then, by axiom $(N2)$, we have  $A \cap V_i \neq \emptyset$ for all $i \in I$ and, being $A$ regular, also $\Int(A) \cap V_i \neq \emptyset$. So, for each $i$ there exists $x_i \in \Int(A) \cap V_i $ and, being $\mathfrak{T}$ a base, there exists $W_k \in \mathfrak{T} : x_i \in W_k \subset V_i$, where $k$ runs in a countable set. Take the smallest regular open set containing $W_k$, $R_k$. We have that $ x_i \in R_k \subset V_i$ because $V_i$, too, is  a regular open set.  On the other side, by $ii)$ we know that there exists $D \in \Sigma(\mathscr{B}')$ such that $ C \subset  D \ll_\delta X \setminus A$. Now let $\mathscr{Z}= (\bigcap_{k=1}^n R_k^\doublewedge) \cap (X \setminus D)^{++}$. We have that $A \in \mathscr{Z} \subset \mathscr{U}$. We can repeat the same procedure even if $A$ is a singleton. 
\end{proof}

\section{Descriptive Strongly Proximal Connectedness}
The concept of \emph{strong proximity}  easily finds applications in several fields. Here we want to present, in particular, connections with \emph{descriptive proximities} and \emph{Vorono\" i regions}. One of the main fields of application for them is image processing.

The theory of descriptive nearness \cite{Peters2014} is usually adopted when dealing with subsets that share some common properties without being spatially close. We talk about \emph{non-abstract points} when points have locations and features that can be measured.  The mentioned theory is particularly relevant when we want to focus on some  of these aspects. For example, if we take a picture element $x$ in a digital image, we can consider grey-level intensity, colour, shape or texture of $x$. We can define an $n$ real valued probe function $\Phi: X \rightarrow \mathbb{R}^n$, where $\Phi(x)= (\phi_1(x),.., \phi_n(x))$ and each $\phi_i$ represents the measurement of a particular feature. So $\Phi(x)$ is a feature vector containing numbers representing feature values extracted from $x$. $\Phi(x)$ is also called \emph{description} of $x$. 

Descriptive nearness is a powerful tool to shift our attention from nearness of sets in a spatial sense to nearness of their features.

\begin{example}
Let $X$ be a bi-dimensional space of picture points and $\Phi: X \rightarrow \mathbb{R}^2$ a description on $X$ defined by $\Phi(x)= ($ color of $x$,gradientAngle in $x)$, where in the first entry we have a value for the color of the picture point $x$, while in the second entry we have the image gradient angle calculated in $x$. It means that to each picture point we can associate a bi-dimensional vector whose entries are represented by axial derivatives of color functions, $\triangledown(f)=\left( \frac{\partial f}{\partial x}, \frac{\partial f}{\partial y} \right)$. Then we can calculate the gradient angle by the formula $\theta = \arctan 2 \left(\frac{\partial f}{\partial x}, \frac{\partial f}{\partial y}\right) $.
\end{example}

\begin{figure}[!ht]
\begin{center}
\includegraphics[width=80mm]{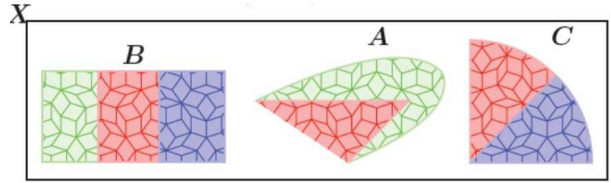}
\caption{$A$ descriptively strongly near $B$}
\label{fig:sndConnectedEx43}
\end{center}
\end{figure}

In~\cite{PetersGuadagni2015connectedness}, we introduced a new kind of connectedness related to strong proximities.

\begin{definition}\label{def:snConnected}
Let $X$ be a topological space and $\sn$ a strong proximity on $X$. We say that $X$ is $\sn-$connected if and only if $X = \bigcup_{i \in I} X_i$, where $I$ is a countable subset of $\mathbb{N}$, $X_i$ and $\Int(X_i)$ are connected for each $i \in I$, and $X_{i-1} \sn X_i$ for each $i \geq 2$.  \qquad \textcolor{blue}{$\blacksquare$}
\end{definition}

\begin{remark}\label{descrsn}
$X_{i-1} \sn X_i$ in Def.~\ref{def:snConnected} can be formulated in descriptive terms. Let $X$ be a set, $\Phi$ a description that maps $X$ to $\mathbb{R}^n$, $\sn$ a strong proximity on $\mathbb{R}^n$ endowed with the Euclidean topology . We say that two subsets $A, B$ are \emph{descriptively strongly near}, and we write
\[
A\ \snd\ B,\ \mbox{if and only if}\ \Phi(A)\ \sn\ \Phi(B).\ \mbox{(descriptive strongly near)} \mbox{\qquad \textcolor{blue}{$\blacksquare$}}
\]
\end{remark}

\begin{example}\label{ex:picturePoints}
Let $X$ be a space of picture points  represented in Fig.~\ref{fig:sndConnectedEx43}.  with red, green or blue colors and let $\Phi: X \rightarrow  \mathbb{R} $ a description on $X$ representing the color of a picture point, where $0$ stands for red (r), $1$ for green (g) and $2$ for blue (b). Suppose the range is endowed with the topology given by $\tau= \{ \emptyset,\{ r,  g\}, \{r,g,b\} \}$. Observe that, choosing such a topology, we pay attention on red and green. Next consider the following strong proximity : $A {\sn} B \Leftrightarrow \Int A \cap \Int B \neq \emptyset$, provided $A$ and $B$ are not singletons; if $A = \{x\}$, then $x \in \Int(B)$, and if $B$ too is a singleton, then $x=y$. Then $A {\sn}_\Phi B$ because $\Phi(A) = \{g,r\}= \Int(\Phi(A))$ and $\Phi(B)= \{r,g,b\}= \Int(\Phi(B))$. Instead $B \not{\sn}_\Phi C$ because $\Phi(C)= \{r,b\}$ and $\Int(\Phi(C))= \emptyset$. \qquad \textcolor{blue}{$\blacksquare$}
\end{example}

\begin{figure}[!ht]
\begin{center}
\includegraphics[width=80mm]{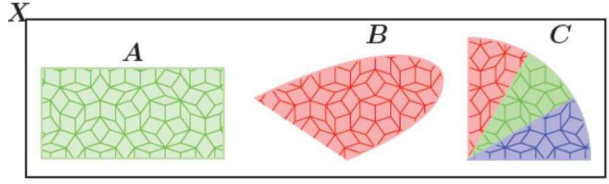}
\caption{$\Phi(A)$ descriptively strongly near $\Phi(C)$}
\label{fig:sndConnectedEx45}
\end{center}
\end{figure}

\begin{definition}
Let $X$ be a topological space and $\sn$ a strong proximity on $X$. We say that $X$ is \emph{descriptively $\sn-$connected} if and only if $X = \bigcup_{i \in I} X_i$, where $I$ is a countable subset of $\mathbb{N}$, $\Phi(X_i)$ and $\Phi( \Int(X_i))$ are connected in the topology on $\mathbb{R}^N$ for each $i \in I$, and $X_{i-1}\ \snd\  X_i$ for each $i \geq 2$.  \qquad \textcolor{blue}{$\blacksquare$}
\end{definition}
\vspace{3mm}

\begin{example}
Let $X$ be a space of picture points with red, green or blue colors represented in Fig.~\ref{fig:sndConnectedEx45}. Take $\Phi,\  \tau $ and $\sn $ as in example \ref{ex:picturePoints}. The space $X= A \cup B \cup C$ is descriptively $\sn-$connected. In fact $\Phi(A)= \{g\}$, $\Int(\Phi(A))= \emptyset$, $\Phi(C)= \{r,g,b\}=\Int(\Phi(C))$, $\Phi(B)= \{r\}$, $\Int(\Phi(B))=\emptyset$ and they are all connected in $\tau$. Furthermore $\Phi(A) \sn \Phi(C)$ and $\Phi(C) \sn \Phi(B)$.\qquad \textcolor{blue}{$\blacksquare$}
\end{example}

\begin{figure}[!ht]
\begin{center}
\includegraphics[width=80mm]{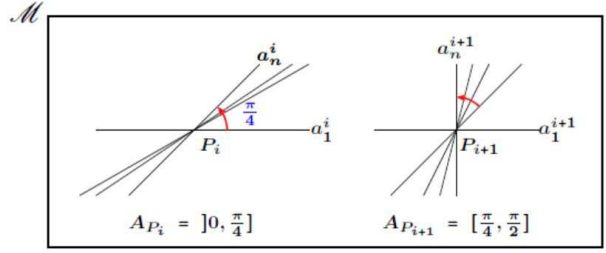}
\caption{$A_{P_{n-1}}$ strongly near $A_{P_{n}}$ for each $n\geq 2$}
\label{fig:sndConnectedEx46}
\end{center}
\end{figure}

\begin{example}\label{ex:curvesManifold}{\bf Curves Manifold}.\\
Let $\mathscr{M}$ be a manifold represented in Fig.~\ref{fig:sndConnectedEx46}. For each point $P_i$ in $\mathscr{M}$ consider a family of curves $\{ \alpha_k^i: k \in K_i\}$ and fix a specific curve $\alpha_1^i$ as reference curve. Take $\theta_{ \{1,k\} }^i$ as the angle between the curves $\alpha_1^i$ and $\alpha_k^i$;  $A_{P_i}= \{ \theta_{ \{1,k\} }^i : k \in K_i \setminus \{1\} \}$. We can talk about descriptive strong connectedness for family of curves. In this case, our space $X$ is represented by all the curves for all the points of $\mathscr{M}$. Our description maps each curve $\alpha_k^i$, for $k \neq 1$, in $\theta_{ \{1,k\} }^i$, and $\alpha_1^i$ in some of the already found values among $\theta_{ \{1,k\} }^i$ with $k \neq 1$ . In particular, we have that the set of all the curves is descriptively $\sn-$connected if we can find a countable subfamily of points $P_n$ such that $A_{P_n} $ is connected and $A_{P_{n-1}} \sn A_{P_n}$ for each $n \geq 2$. To better understand look at figure. We have $A_{P_i}= ]0, \frac{\pi}{4}]$ and $A_{P_{i+1}}= [\frac{\pi}{4},\frac{\pi}{2}]$. They are connected and, if we take $A \sn B \Leftrightarrow A \cap B \neq \emptyset$, they are also strongly near. It means that, fixed $\alpha_1^i$ and $\alpha_1^{i+1}$ there is at least one curve through $P_i$ and one through $P_{i+1}$ such that they form the same angle with $\alpha_1^i$ and $\alpha_1^{i+1}$ respectively. We could require more choosing a stronger strong proximity. In the previous way we obtained a sort of angle connectedness for families of curves. \qquad \textcolor{blue}{$\blacksquare$}
\end{example}

\section{Proximal and strongly proximal manifolds}
Suppose that $\mathscr{M}$ is a topological space.  $\mathscr{M}$ is a \emph{manifold} of dimension $n$, provided it is Hausdorff, second countable and locally Euclidean of dimension $n$ so that each point $p$ in $\mathscr{M}$ has a neighbourhood (an open set containing $p$) of $U$ that is homeomorphic to an open set in $\mathbb{R}^n$.  Let $\varphi:U\longrightarrow \mathbb{R}^n$ be a homeomorphism on the image.  A \emph{chart} on $\mathscr{M}$ is a pair $(U,\varphi)$.  When the meaning is clear from the context, we write chart $U$ instead of $(U,\varphi)$.  An \emph{atlas} $\mathcal{A}$ for manifold $\mathscr{M}$ is a collection of charts whose domain covers $\mathscr{M}$.  Given a pair of charts $(U,\varphi),(V,\psi)$, the composite map $\psi\circ\varphi^{-1}:\varphi(U\cap V)\longrightarrow \psi(U\cap V)$ is called a transition map from $\varphi$ to $\psi$. 

A pair of charts is \emph{smoothly compatible}, provided $U\cap V \neq\emptyset$ and the transition map $\psi\circ\varphi^{-1}$ is a $C^\infty-$diffeomorphism on $\varphi(U \cap V)$.  An atlas $\mathcal{A}$ is \emph{smoothly compatible}, provided any pair of charts in $\mathcal{A}$ is smoothly compatible~\cite[\S 1, p. 12]{Lee2013smoothManifold}.  By replacing the requirement that charts be smoothly compatible
 with the weaker requirement that each transition map $\psi\circ\varphi^{-1}$ and its inverse are $C^r-$differentiable, $\mathscr{M}$ is called a $C^r-$manifold.
  
Suppose that $\mathscr{M}$ is an $n-$dimensional $C^r-$manifold. We can endow it with a proximity that is strictly connected with its structure. For example, if $\mathcal{A}= \{(U_i, \phi_i): i \in I \} $ is an atlas on $\mathscr{M}$, we can define a proximity on $\mathcal{A} \times \mathcal{A}$ in the following way:

\begin{center}
$U_i\ \delta\ U_j $\\
$ \Leftrightarrow$\\
$ \exists C \subseteq U_i, D \subseteq U_j \ f: \phi_i(C) \rightarrow \phi_j(D)$,  
\end{center}
such that $f$ is  $C^r-$diffeomorphism.

\begin{theorem}\label{proxmanif}
The relation $\delta$ is a proximity on $\mathcal{A} \times \mathcal{A}$.
\end{theorem}
\begin{proof}
$1)$ We have that $\emptyset \not\delta U_i , \ \forall U_i \in \mathcal{A}$ because $\emptyset$ is not a domain for any chart. $2)$ Symmetry is obvious. $3) \ U_i \cap U_j \neq \emptyset \Rightarrow U_i\ \delta\ U_j$, since we can consider the transition maps on  $U_i \cap U_j$. $4)$ We want to show that if $U_j \cup U_k = U_h \in \mathcal{A}, \ U_i \delta (U_j \cup U_k) \Leftrightarrow U_i\ \delta\ U_j \hbox{ or } U_i\ \delta\ U_k $.\\ 
\noindent $\stackrel{4)}{\Leftarrow}$: Suppose that $U_i\ \delta\ U_j$. So there exist $C \subseteq U_i, D \subseteq U_j \ f: \phi_i(C) \rightarrow \phi_j(D)$  s.t. $f$ is a $C^r-$diffeomorphism. But, being $D \subseteq U_j \cup U_k$, we have also $g: \phi_j(D) \rightarrow \phi_h(D)$ that is a $C^r-$diffeomorphism. Hence, by composing $f$ and $g$ we obtain $U_i\ \delta\ (U_j \cap U_k)$.\\ 
\noindent $\stackrel{4)}{\Rightarrow}$: Suppose $U_i\ \delta\ (U_j \cup U_k)$. Then there exist $C \subseteq U_i, E \subseteq U_h= U_j \cup U_k \ f: \phi_i(C) \rightarrow \phi_h(E)$  such that $f$ is  a $C^r$-diffeomorphism. We can assume that $E \cap U_j \neq \emptyset$. So we have a transition map $g: \phi_h(E \cap U_j) \rightarrow \phi_j(E \cap U_j) $ that is a $C^r-$diffeomorphism. Now we can take the restriction of $f$, $f^*$, that is an homeomorphism onto $\phi_h(E \cap U_j)$. By composing $f^* $ and $g$ we obtain the desired result.
\end{proof}

On a manifold, it is possible to define also a stronger kind of proximity,called a \textit{manifold strong proximity}. As before, take an atlas $\mathcal{A}= \{(U_i, \phi_i): i \in I \} $. 

\begin{definition}
Let ${\sn}_\mathcal{A}$ be a relation on $\mathcal{A} \times \mathcal{A}$. It is called \emph{manifold strong proximity}, if the following axioms hold:
\begin{enumerate}
\item[$(M0)$] $\emptyset \not{\sn}_\mathcal{A} U_i \forall i \in I$,
\item[$(M1)$] $U_i {\sn}_\mathcal{A} U_j \Leftrightarrow U_j {\sn}_\mathcal{A} U_i \ \forall i,j \in I $
\item[$(M2)$] $U_i {\sn}_\mathcal{A} U_i \Rightarrow \phi_i(U_i) \cap \phi_j(U_j) \neq \emptyset$,
\item[$(M3)$] If $\{U_h\}_{h\in H \subseteq I}$ is an arbitrary family of domains in $\mathcal{A}$ and  $U_i {\sn}_\mathcal{A} U_j$ for some $j \in H \setminus \{i\} \subseteq I$, then  $ U_i {\sn}_\mathcal{A} \bigcup_{h \in H \setminus \{i\}} U_h$.
\item[$(M4)$] $\Int(\phi_i(U_i)) \cap \Int(\phi_j(U_j))\neq \emptyset \Rightarrow U_i {{\sn}_\mathcal{A}} U_j$
\end{enumerate}
\end{definition}

Define the following relation on $\mathcal{A} \times \mathcal{A}$:
\[
U_i\ {\sn}_\mathcal{A}\ U_j \Leftrightarrow \phi_i(U_i) \cap \phi_j(U_j) \neq \emptyset. 
\]

That is, chart $U_i$ is strongly near chart $U_j$, if and only if the chart descriptions $\phi_i(U_i) \cap \phi_j(U_j)$ have nonempty intersection.  Moreover, if $U_j \cup U_k $ is not a domain in $\mathcal{A}$, define $U_i {\sn}_{\mathcal{A}} (U_j \cup U_k) \Leftrightarrow \phi_i(U_i) \cap (\phi_j(U_j) \cup \phi_k(U_k)) \neq \emptyset $. $(*)$

\begin{theorem}
The relation ${\sn}_{\mathcal{A}}$ is a manifold strong proximity on $\mathcal{A} \times \mathcal{A}$, if $U_i \cup U_j$ is not a domain in $\mathcal{A}$ for all $i$ and $j$ with $i \neq j$.
\end{theorem}
\begin{proof}
$M0) $ That $\emptyset \not{\sn}_\mathcal{A} U_i$ for each $i\in I$ is straightforward. Moreover, if $\mathscr{M}$ is a domain, we do not need any other domain. $M1)$ Symmetry is obvious. $M2)$ The descriptive form of $A \delta B \Rightarrow A \cap B \neq \emptyset$ holds by definition. $M3)$ This holds because we know that $U_i \cup U_j$ is not a domain in $\mathcal{A}$ for all $i$ and $j$ with $i \neq j$. So we refer to $(*)$. $M4)$ This holds being $U_i$ and $U_j$ open sets and the charts homeomorphisms.
\end{proof}

In terms of the proximity relation $\delta$ on $\mathcal{A}\times \mathcal{A}$ from Theorem~\ref{proxmanif}, we obtain the following result.

\begin{theorem}\label{thm:atlasProximity}
Let $U_i,U_j\in \mathcal{A}$ be charts in manifold atlas $\mathcal{A}$.  $U_i\ {{\sn}_\mathcal{A}}\ U_j \Rightarrow U_i\ \delta\ U_j$.
 \end{theorem}
 \begin{proof}
If the intersection $ \phi_i(U_i) \cap \phi_j(U_j)$ is non-empty, we can take that part of $U_i$ that is mapped in $ \phi_i(U_i) \cap \phi_j(U_j)$, and the same with $U_j$. On the intersection we can take the identity map that is obviously a $C^r-$diffeomorphism.
\end{proof}

\begin{remark}
Observe that is particularly interesting to see that a manifold is descriptively ${{\sn}_\mathcal{A}}-$connected if we have on it an atlas composed by a countable number of connected domains such that $ \phi_i(U_i) \cap \phi_{i+1}(U_{i+1}) \neq \emptyset, \ \forall i \in I$.
\end{remark}

\begin{example}
A simple example of descriptively  ${{\sn}_\mathcal{A}}-$connected  manifold is $S^1$ with the stereographic projection atlas. In fact in this case we have two charts:
\begin{center}
$\phi_1: S^1 \setminus\{N\} \rightarrow \mathbb{R}$, $\phi_1(x,y)= \dfrac{1}{1-y}$,\\
$\phi_2: S^1 \setminus\{S\} \rightarrow \mathbb{R}$, $\phi_2(x,y)= \dfrac{1}{1+y}$,\\
\end{center}

\noindent where $N \equiv (0,1) $ is the north, and $S \equiv (0,-1)$ is the south.
We have that the domain are homeomorphic to the whole $\mathbb{R}$, so $\phi_1(S^1 \setminus \{N\}) \cap \phi_2(S^1 \setminus \{S\})\neq \emptyset $.
 
\end{example}
 
\begin{figure}[!ht]
\begin{center}
\includegraphics[width=35mm]{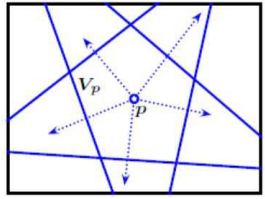}
\end{center}
\caption[]{$V_p$ = Intersection of closed half-planes}
\label{fig:convexPolygon}
\end{figure}

\section{Strongly proximal Vorono\" i regions}

A \emph{Vorono\" i diagram} represents a tessellation of the plane by convex polygons. It is generated by $n$ site points and each polygon contains exactly one of these points. In each region there are points that are closer to its generating point than to any other. \emph{Vorono\" i diagrams} were introduced by \emph{Ren\' e Descartes} (1667)  looking at the influence regions of stars. They were studied also by Dirichlet (1850) and Vorono\" i (1907), who extended the study to higher dimensions. 

To construct a Vorono\" i diagram, we have to start from a finite number of points. Consider a set $S$ of $n$ points in a finite-dimensional normed vector space $(X, \norm{\cdot})$. We call $S$ the \emph{generating set}. The Vorono\" i diagram based on $S$ is constructed by taking for each point of $S$ the intersection of suitable half planes. Take $p \in S $ and let $H_{pq}$ be the closed half plane of points at least as close to $p$ as to $q \in S \setminus \{p\}$ given by 
\[ 
H_{pq}= \{ x \in X : \norm{x-p} \leq \norm{x-q} \} .
\]
\noindent The intersection of all the half planes for $q \in S \setminus \{p\}$ gives the \emph{Vorono\" i region} $V_p$ of $p$:
\[ 
V_p = \bigcap_{q \in S \setminus \{p\}}	H_{pq}. 
\]
Vorono\"{i} regions are named after the Ukrainian mathematician Georgy Vorono\"{i}~\cite{Voronoi1903,Voronoi1907,Voronoi1908}.  The simplifying notation $V(p)$ is sometimes used instead of $V_p$, when $p$ is an expression such as $a_i$ for an indexed site. 

\begin{lemma}\label{lem:convexity}~{\rm \cite[\S 2.1, p. 9]{Edelsbrunner2014}} The intersection of convex sets is convex.
\end{lemma}
\begin{proof}
Let $A,B\subset \mathbb{R}^2$ be convex sets and let $K = A\cap B$.  For every pair points $x,y\in K$, the line segment $\overline{xy}$ connecting $x$ and $y$ belongs to $K$, since this property holds for all points in $A$ and $B$.   Hence, $K$ is convex.
\end{proof}

Since a Vorono\" i region is the intersection of closed half planes, each \emph{Vorono\" i region } is a closed convex polygon (see, {\em e.g.}, Fig.~\ref{fig:convexPolygon}).

\begin{remark} \rm
The Vorono\"{i} region $V_p$ depicted as the intersection of finitely many closed half planes in Fig.~\ref{fig:convexPolygon} is a variation of the representation of a Vorono\"{i} region in the monograph by H. Edelsbrunner~{\rm \cite[\S 2.1, p. 10]{Edelsbrunner2014}}, where each half plane is defined by its outward directed normal vector.  The rays from $p$ and perpendicular to the sides of $V_p$ are comparable to the lines leading from the center of the convex polygon in G.L. Dirichlet's drawing~\cite[\S 3, p. 216]{Dirichlet1850}.
\qquad \textcolor{black}{$\blacksquare$}
\end{remark}

We want to define a strong proximity acting on Vorono\" i regions. We say that two Vorono\" i regions are strongly near, and we write $V_p\ \sn\ V_q$, if and only if they share more than one point.  

\begin{theorem}\label{thm:VoronoiRegions}
Let $(X, \norm{\cdot})$ be a finite-dimensional normed vector space and $S$ a collection of points in $X$. The relation defined by saying $V_p \sn V_q$ if and only if they share more than one point is a strong proximity on $\mathscr{V}(S)$, the class of Vorono\" i regions generated by $S$.
\end{theorem}
\begin{proof}
Axioms $N0)$ through $N3)$ are easily verified. Axiom $N4)$ holds, since the intersection of the interiors is always empty. That is,  
\[
V_p\ \notsn\ V_q\ \Rightarrow\ \Int V_p\cap \Int V_q = \emptyset, V_p \neq V_q.
\]
Axiom $N5)-N6)$ hold because there are no points in common among the interiors of the Vorono\" i regions.
\end{proof}

Theorem~\ref{thm:VoronoiRegions} is illustrated in Example~\ref{ex:VoronoiRegions}.

\begin{figure}[!ht]
\begin{center}
\includegraphics[width=35mm]{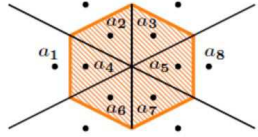}
\end{center}
\caption[]{Vorono\" i Regions $V_{a_i}, i\in \left\{1,2,3,4,5,6,7,8\right\}$}
\label{fig:PartialVoronoiDiagram}
\end{figure} 

\begin{example}\label{ex:VoronoiRegions}
Let $X$ be a space covered with a Vorono\"i diagram $\mathscr{V}(S)$, $S$, a set of sites.
A partial view of $\mathscr{V}(S)$ is shown in Fig.~\ref{fig:PartialVoronoiDiagram}, where 
\begin{align*}
V_{a_i} \in \mathscr{V}(S), i &\in \left\{1,2,3,4,5,6,7,8\right\}.
\end {align*} 
From Theorem~\ref{thm:VoronoiRegions}, observe that
\[
V_{a_4}\ \sn\ V_{a_2},\ V_{a_4}\ \sn\ V_{a_6}\ \mbox{and}\ V_{a_4}\ \notsn\ V_{a_5},\ V_{a_2}\ \notsn\ V_{a_5},
\ V_{a_6}\ \notsn\ V_{a_5},
\]
since $\left\{V_{a_2},V_{a_4}\right\},\left\{V_{a_4},V_{a_6}\right\}$ have a common edge.  Further, $V_{a_2}, V_{a_4}, V_{a_6}$ are not strongly near $V_{a_5}$. $V_{a_2}, V_{a_4}, V_{a_6}$ share only one point with $V_{a_5}$.   Similarly, 
\[
V_{a_5}\ \sn\ V_{a_3}, V_{a_5}\ \sn\ V_{a_7}, V_{a_5}\ \sn\ V_{a_8},
\]
since, taken pairwise, these Vorono\" i regions have a common edge.   There are also Vorono\" i regions in Fig.~\ref{fig:PartialVoronoiDiagram} that are near but not strongly near, {\em e.g.}, $V_{a_3}\ \notsn\ V_{a_6}, V_{a_7}\ \notsn\ V_{a_2}$.  
\qquad \textcolor{blue}{$\blacksquare$}
\end{example}

From Theorem~\ref{thm:VoronoiRegions}, we can define a strongly hit and miss hypertopology, $\tau^\doublewedge$, on the space of Vorono\" i regions generated by $S$, $\mathscr{V}(S)$, to which we add the empty set, \cite{PetersGuadagni2015stronglyNear}. The hypertopology $\tau^\doublewedge$ has as subbase the elements of the following form:\\
\begin{itemize}
\item $ \Int(V_p)^\doublewedge = \{ V_q \in \mathscr{V}(S): V_q \ \sn \  \Int(V_p)\} = \{ V_q \in \mathscr{V}(S): V_q \sn V_p \}$,
\item $\quad V_s^{+}= \{ V_q \in \mathscr{V}(S): V_q \cap V_s = \emptyset \}$,
\end{itemize}
\noindent where Vorono\" i regions $V_p, V_s \in \mathscr{V}(S)$.

\begin{theorem}
Let $(X, \norm{\cdot})$ be a finite-dimensional normed vector space and $S$ a collection of points in $X$. For any $p \in S$ let $\{a_i\}_{i \in I}$ the family of points in $S$ such that $V_{a_i} \sn V_p$, and $\{b_j\}_{j \in J} $ the family of points in $S$ such that $V_{b_j} \cap V_p = \emptyset$. Then $\mathscr{A}= (\bigcap_{i=1}^n \Int(V_{a_i})^\doublewedge) \cap (\bigcap_{j=1}^m V_{b_j}^+)$ is the smallest open set in $\tau^\doublewedge$ containing $V_p$.
\end{theorem}
\begin{proof}
Suppose that $\mathscr{B}$ is an open set in $\tau^\doublewedge$ such that $V_p \in \mathscr{B} \subseteq \mathscr{A}$. Then $\mathscr{B} = \mathscr{A} \cap (\bigcap_{k=1}^r \Int(V_{c_k})^\doublewedge) \cap (\bigcap_{h=1}^s V_{d_h}^+)$, where $c_1,..,c_r, d_1,.., d_s \in S$. It means that $V_p \sn V_{c_k}$ for each $k=1,..,r$ and $V_p \cap V_{d_h}= \emptyset$ for each $h=1,..,s$. So, by the hypothesis, we have that each $c_k$ has to coincide with some point in $\{a_i\}_{i \in I}$ and each $d_h$ has to coincide with some point in $\{b_j\}_{j \in J}$. That is $\mathscr{B} = \mathscr{A}$.
\end{proof}

\begin{figure}[!ht]
\begin{center}
\includegraphics[width=35mm]{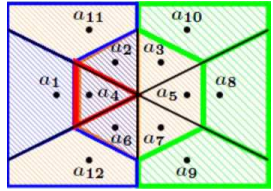}
\end{center}
\caption[]{Smallest open set containing Voron\" i region of a4}
\label{fig:VoronoiDiagram2}
\end{figure}

\begin{figure}[!ht]
\begin{center}
\includegraphics[width=50mm]{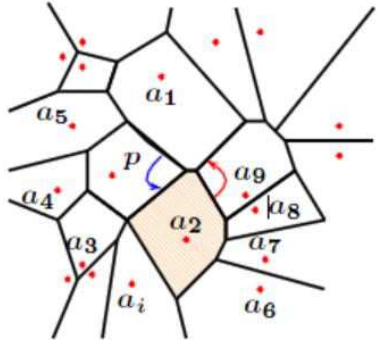}
\end{center}
\caption{Vorono\" i diagram with respect to Theorem 6.7}
\label{fig:VoronoiDiagram3}
\end{figure}

\begin{example}
Consider the situation in Fig.~\ref{fig:VoronoiDiagram2}. Take, for example, the Vorono\" i region $V_{a_4}$. The smallest open set in $\tau^\doublewedge$ containing $V_{a_4}$ is given by 
\[
\mathscr{A} = (\bigcap_{i=1,2,4,6} \Int(V_{a_i})^\doublewedge) \cap (\bigcap_{q=8,9,10} V_{q}^+). \mbox{\qquad \textcolor{blue}{$\blacksquare$}}
\]   
\end{example}

\begin{theorem}\label{region}
Let $(X, \norm{\cdot})$ be a finite-dimensional normed vector space and $S$ a collection of points in $X$. If $p \in S$ and $\mathscr{A}$ is the smallest open set containing $V_p$, then $\mathscr{A}$ cannot contain any other region in $\mathscr{V}(S)$.
\end{theorem}
\begin{proof}
Let $\mathscr{A}$ be $\mathscr{A}= (\bigcap_{i=1}^n \Int(V_{a_i})^\doublewedge) \cap (\bigcap_{j=1}^m V_{b_j}^+)$ and suppose it is the smallest open set containing the Vorono\" i region, $V_p$, of a point $p$ . If there is another region $V_q \in \mathscr{A}$, then $V_q \sn V_{a_i}$ for all $i=1,..,n$. Suppose $V_{a_i}$ are indexed in such a way that each $V_{a_i}$ has non-empty intersection with the next one $V_{a_{i+1}}$ for $i=1,.., n-1$, and $V_n$ has non-empty intersection with $V_{a_1}$. This is possible because they define $V_p$.  Consider now $V_{a_1}$ and $V_{a_2}$. Since $V_p$ is convex, $V_{a_1}$ and $V_{a_2}$ have to form a convex angle, and because also $V_{a_1}$ and $V_{a_2}$ are convex, they can intersect at most in an edge. But also $V_q$ is convex and it is delimited by $V_{a_1}$ and $V_{a_2}$. So either $V_q$ has the same convex angle as $V_p$, or it can have a different convex angle situated on the opposite side, that is outside $H_{pa_1} \cap H_{pa_2}$, intersection of half planes (see, {\em e.g.}, Fig.~\ref{fig:VoronoiDiagram3}). Suppose it is verified this last situation. We know also that $V_{a_3}$ delimits $V_q$. By the last supposition it would mean that we should take the convex angle formed by $V_{a_2}$ and $V_{a_3}$ situated outside $H_{pa_2}\cap H_{pa_3}$. By continuing in this way for all the points $a_1,.., a_n$ we obtain an absurd by the convexity of all regions. So we have to consider necessarily the same convex angles as $V_p$ and we obtain that $V_q=V_p$.

\end{proof}

\section{Proximal Vorono\" i Manifolds, Atlases and Charts}
Let $\mathcal{M}$ be a manifold, that is a topological space which is Hausdorff, second countable, locally Euclidean of dimension $n$. This means that for each point there is a neighbourhood $U$ of $\mathcal{M}$ with a homeomorphism $\varphi: U\longrightarrow \hat{U}= \varphi(U)\subseteq\mathbb{R}^n$.  $\mathcal{M}$ is a Vorono\" i manifold, provided $\varphi(U)$ is a Vorono\" i diagram.  The pair $(U,\varphi)$ is called a \emph{Vorono\" i chart} on $\mathcal{M}$.  The collection $\mathcal{A}$ of all Vorono\" i charts on $\mathcal{M}$ is called a \emph{Vorono\" i atlas}.

Let $\mathcal{M}_1,\mathcal{M}_2$ be Vorono\" i manifolds and let $S_1\subset \mathcal{M}_1,S_2\subset \mathcal{M}_2$ be nhbds of points in $\mathcal{M}_1$ and $\mathcal{M}_2$ respectively, $\varphi, \psi$ homeomorphisms from $ S_1, S_2$ to subsets $\varphi(S_1), \psi(S_2)\subseteq\mathbb{R}^n$ such that $\varphi(S_1), \psi(S_2)$ are Vorono\" i diagrams.     From what has been observed about manifolds, we make the following observations. Define \\
\[ 
\mathcal{M}_1\ \sn\ \mathcal{M}_2\Leftrightarrow \exists (S_1, \varphi), (S_2, \psi): \varphi(S_1)\ \sn\ \psi(S_2) \hbox{ in the sense of thm. \ref{thm:VoronoiRegions}}.
\]

\begin{example}\label{ex:nearCharts}
Let $\mathcal{M}_1,\mathcal{M}_2$ be Vorono\" i manifolds in the plane.  From Fig.~\ref{fig:VoronoiDiagram3}, let
\begin{align*}
\mathcal{M}_1 &= \text{the portion of the plane containing the regions associated with}\ a_1,a_3,a_4,a_5,p\\
\mathcal{M}_2 &= \text{the portion of the plane containing the regions associated with}\ a_2,a_6,a_{7},a_{8},a_{9}\\
S_1 &= \text{the interior of the portion of the plane containing the regions associated with}\ p, a_1 \\
S_2  &= \text{the interior of the portion of the plane containing the regions associated with}\ a_2, a_8\\
\varphi(S_1) &= \mathscr{V}(S_1)\ \mbox{(Vorono\" i diagram)}, V_p\in \mathscr{V}(S_1).\\
\psi(S_2) &= \mathscr{V}(S_2)\ \mbox{(Vorono\" i diagram)}, V_{a_2}\in \mathscr{V}(S_1).
\end{align*}
In this simple case, the homeomorphisms correspond to the identity map. In Fig.~\ref{fig:VoronoiDiagram3}, $\mathscr{V}(S_1),\mathscr{V}(S_2)$ share the edge between Vorono\" i regions $V_p$ and $V_{a_2}$.  Hence, 
$\varphi(S_1)\ \sn\ \psi(S_2)$. So $\mathcal{M}_1\ \sn\ \mathcal{M}_2$. \qquad \textcolor{blue}{$\blacksquare$} 
\end{example}
In terms of descriptively near manifolds $\mathcal{M}_1,\mathcal{M}_2, S_1\subset \mathcal{M}_1, S_2\subset \mathcal{M}_2$ with corresponding descriptively near charts $(S_1,\varphi),(S_2,\psi)$, we have
\[
\mathcal{M}_1\ \snd\ \mathcal{M}_2\Leftrightarrow \exists (S_1, \varphi), (S_2, \psi): \varphi(S_1)\ \snd\ \psi(S_2), \hbox{ in the sense of page \pageref{descrsn}}
\]
where $\Phi: \mathcal{M}_1 \cup \mathcal{M}_2 \rightarrow \mathbb{R}^n$.

\begin{example}\label{ex:descriptivelyConnectedVoronoiManifolds}
Continuing Example~\ref{ex:nearCharts}, assume
\begin{align*}
x &\in \varphi(S_1), y \in \psi(S_2).\\
\Phi(x) &= (\mbox{colour of}\ x, \theta_x\ \mbox{gradient angle}),\ \mbox{feature vector for}\ x.\\
\Phi(y) &= (\mbox{colour of}\ y, \theta_y\ \mbox{gradient angle}),\ \mbox{feature vector for}\ y.
\end{align*}
Assume $x,y$ have matching feature vectors, then
\[
\mathcal{M}_1\ \snd\ \mathcal{M}_2  \Leftrightarrow \exists (S_1, \varphi), (S_2, \psi): \Phi(\varphi(S_1))\ \sn\ \Phi(\psi(S_2)). \mbox{\qquad \textcolor{blue}{$\blacksquare$}}
\]
\end{example}

Let $\mathcal{A}_1 = \left\{(U_i,\varphi_i): i\in\mathbb{N}^+\right\},\mathcal{A}_2 = \left\{(V_j,\psi_j): j\in\mathbb{N}^+\right\}$ be atlases on smooth manifolds $\mathcal{M}_1,\mathcal{M}_2$, respectively, $\hat{U}_i= \varphi_i(U_i), \hat{V}_j= \psi_j(V_j)$, and define the \emph{descriptive intersection of the disjoint charts} by
\[
\hat{U_i}\ \mathop{\cap}\limits_{\Phi,\hat{ \mathcal{A}}} \ \hat{V_j} = \left\{ x \in \hat{U}_i \cup \hat{V}_j: \Phi(x)\in \Phi(\hat{U}_i) , \ \Phi(x) \in \Phi(\hat{V_j}) \right\}.
\]

\noindent Then define the relation $\sndatlases$  on $\mathcal{A}_1\times\mathcal{A}_2$ by
\[
U_i\ \sndatlases\ V_j \Leftrightarrow \hat{U}_i\ \mathop{\cap}\limits_{\Phi, \hat {\mathcal{A}}}\ \hat{U}_j\neq\emptyset.
\]

\begin{theorem}\label{thm:nearAtlases}
Let $\mathcal{A}_1 = \left\{(U_i,\varphi_i): i\in\mathbb{N}^+\right\},\mathcal{A}_2 = \left\{(V_j,\psi_j): j\in\mathbb{N}^+\right\}$ be atlases on smooth manifolds $\mathcal{M}_1,\mathcal{M}_2$, respectively.   Then 
\begin{enumerate}
\item $\hat{U}_i\ \snd\ \hat{V}_j \Rightarrow U_i \sndatlases V_j$

\item $\mathcal{M}_1\ \snd\ \mathcal{M}_2$  $\Rightarrow \exists (U_i, \varphi_i)\in \mathcal{A}_1, (V_j, \psi_j) \in \mathcal{A}_2 : U_i \sndatlases V_j$
\end{enumerate}
\end{theorem}
\begin{proof}$\mbox{}$\\
$(1)$: Suppose $\hat{U}_i\ \snd\ \hat{V}_j$. By definition of descriptive strong nearness (Remark~\ref{descrsn}), we have $\Phi(\hat{U}_i) \sn \Phi(\hat{V}_j)$. So $\Phi(\hat{U}_i) \cap \Phi(\hat{V}_j) \neq \emptyset$. This means that $\hat{U}_i\ \mathop{\cap}\limits_{\Phi, \hat {\mathcal{A}}}\ \hat{U}_j\neq\emptyset$. Hence $U_i \sndatlases V_j$. \\
$(2)$: We know that $\mathcal{M}_1\ \snd\ \mathcal{M}_2$ . So there exist $(U_i, \varphi_i) \in \mathcal{A}_1, (V_j, \psi_j) \in \mathcal{A}_2:  \varphi(U_i)\ \sn\ \psi(V_j) $. Hence, from $(1)$, we have that there exist $(U_i, \varphi_i) \in \mathcal{A}_1, (V_j, \psi_j) \in \mathcal{A}_2$ such that $U_i \sndatlases V_j   $.
\end{proof}

\begin{remark}
Observe that the converse of $(1)$ is not in general true. In fact we could have $\Phi(\hat{U}_i) \cap \Phi(\hat{V}_j) \neq \emptyset$ but $\Phi(\hat{U}_i) \not{\sn} \Phi(\hat{V}_j) $. This would mean $U_i \sndatlases V_j$ but $\hat{U}_i\ \notsnd\ \hat{V}_j$.
\end{remark}


\end{document}